\documentclass[microtype]{gtpart}   
\usepackage{colonequals}
\usepackage[utf8x]{inputenx}
\usepackage[all, cmtip, arrow]{xy}
\title{Real Equivariant Bordism for elementary abelian $2$--groups}
\author{Moritz Firsching}
 \givenname{Moritz}
 \surname{Firsching}
 \address{Institut für Mathematik, FU Berlin\\Arnimallee 2\\14195 Berlin\\ Germany}
 \email{firsching@math.fu-berlin.de}
 \urladdr{http://page.mi.fu-berlin.de/moritz/}

\keyword{Real Equivariant Bordism}
\subject{primary}{msc2010}{57R85}
\subject{secondary}{msc2010}{55N91}

\arxivreference{}
\arxivpassword{}
\volumenumber{}
\issuenumber{}
\publicationyear{}
\papernumber{}
\startpage{}
\endpage{}
\doi{}
\MR{}
\Zbl{}
\received{}
\revised{}
\accepted{}
\published{}
\publishedonline{}
\proposed{}
\seconded{}
\corresponding{}
\editor{}
\version{}

\newcommand{\sma}{\wedge}
\DeclareMathOperator{\pt}{pt}
\DeclareMathOperator*{\colim}{colim}
\DeclareMathOperator{\id}{id}
\DeclareMathOperator{\Map}{Map}
\DeclareMathOperator{\im}{im}
\DeclareMathOperator{\PT}{PT}

\newtheorem{theorem}{Theorem}[section]
\newtheorem{lemma}[theorem]{Lemma}
\newtheorem{propo}[theorem]{Proposition}
\newtheorem{corollary}[theorem]{Corollary}
\newtheorem*{hauptsatz}{Theorem \ref{coro1}}

\theoremstyle{definition}
\newtheorem{definition}[theorem]{Definition}

\theoremstyle{remark}
\newtheorem{remark}[theorem]{Remark}

\def\co{\colon\thinspace}
\date{\today}
\begin{document}\begin{abstract}
 We give a description of real equivariant bordism for the group $G=\mathbb{Z}/2\times\dots\times\mathbb{Z}/2$, 
which is similar to the description of \emph{complex} equivariant bordism for the group $S^1\times \dots\times S^1$ given by Hanke in \cite[Theorem 1]{H05}.   
\end{abstract}
\maketitle
 
\section{Introduction}
Non-equivariantly, the  Pontryagin--Thom construction (see the classical paper by Thom \cite{T54}) yields an isomorphism between (non-equivariant) real geometric bordism $\mathfrak{N}_*$ and the groups $MO_*$ associated to the (non-equivariant) Thom spectrum.  
For a compact Lie group $G$, Conner and Floyd defined \emph{equivariant} real geometric bordism  $\mathfrak{N}^G_*$ (see \cite[Section 19]{CF79}). 
The study of an equivariant analogue of the Thom spectrum and groups associated to it leads to the definition of \emph{equivariant homotopic} bordism $MO_*^G.$ 
It's possible to define a Pontryagin--Thom map between these groups, but this fails to be an isomorphism if $G$ is a non-trivial group due to a lack of equivariant transversality. 
It it not known, whether the equivariant Pontryagin--Thom map is a monomorphism for all groups; however for $G=\mathbb{Z}/2\times\dots\times\mathbb{Z}/2,$ injectivity was shown by tom Dieck, see \cite[Theorem 1]{tD71}.

Sinha investigates $\mathfrak{N}_*^{\mathbb{Z}/2}$  and describes it as a subring of $MO_*^{\mathbb{Z}/2},$ generated by certain elements
that are images of geometric bordism classes under the equivariant Pontryagin--Thom map, see \cite{S02}. He also describes the quotient $\mathfrak{N}^{\mathbb{Z}/2}_*$--module $MO_*^{\mathbb{Z}/2}/\mathfrak{N}_*^{\mathbb{Z}/2}$ 
which can be interpreted as transversality obstructions.
Instead of the group $\mathbb{Z}/2$, which Sinha considers, we look more generally at real equivariant bordism for groups of the form $\mathbb{Z}/2\times\dots\times\mathbb{Z}/2$. 
Our main result is the following:
\begin{hauptsatz}For $G=\mathbb{Z}/2\times\dots\times \mathbb{Z}/2$, there is a pullback square 
 \[\xymatrix{\mathfrak{N}_*^G\ar[d]\ar[r]&MO_*[e^{-1}_V,Y_{d,V}]\ar[d]\\
                 MO_*^G\ar[r] & MO_*[e_V,e^{-1}_V,Y_{d,V}] } \]
with all maps injective, for certain indeterminates $e_V, e^{-1}_V$ and $Y_{d,V}$, where $V$ runs through a complete set $J$ of representatives of isomorphism classes of non-trivial irreducible representations of $G$ and $d>1.$
\end{hauptsatz}
This theorem is an analogue of a result in complex equivariant bordism for $G=S^1\times\dots\times S^1$ proved by Hanke, see \cite[Theorem 1]{H05}. In our proof we follows Hanke's argument
closely and use the same techniques, most of which have already been employed in this context in a paper by Sinha, see \cite{S01}. 
The methods of proof include the use of families of subgroups and the analysis of normal data around fixed sets. 
Some considerations become easier in view of the fact that we are considering a finite group, instead of the compact, but infinite, group 
$S^1\times\dots\times S^1.$ 
In order to define homotopic equivariant bordism, we use, but don't discuss in much detail, the notion of complete $G$--universes and equivariant homotopy and homology theory and give detailed references instead. 

In the next section, we review real equivariant bordism and its basic properties. Section \ref{drei} is occupied with the proof of Theorem \ref{thee}, carefully defining all necessary maps first. 
Our article is based on \cite{F11} and we  strive to give elaborate definitions and proofs. However when a result can be found already well-documented in the literature, we refer the reader to it.

\subsection*{Acknowledgments}
I am grateful to Jens Hornbostel for drawing my attention to this problem and answering many questions and to Bernhard Hanke for encouraging me to write this paper. I owe a major debt of gratitude to Peter Landweber for numerous valuable comments and suggestions. 

\section{Real equivariant bordism}
\subsection{Geometric real equivariant bordism}
Let $G$ be a compact Lie group. 
Let's recall the basic notions of geometric real equivariant bordism, which was described in depth first by Conner and Floyd, see  \cite{CF79} and \cite{C79}.

All manifolds we consider in connection with bordism groups of any kind are assumed to be smooth and compact. Group actions on these manifolds are required to be smooth.  
A \emph{singular $G$--manifold} over a pair of $G$--spaces $(X,A)$ is a $G$--manifold $M$ with (possibly empty) boundary together with a $G$--equivariant map \[f\co(M,\partial M)\rightarrow (X,A).\]
Two singular $n$--dimensional $G$--manifolds, $(M_1,f_1)$ and $(M_2,f_2)$, over $(X,A)$ are \emph{bordant} if there is a $(n+1)$--dimensional $G$--manifold $W$ with 
two $G$--submanifolds of codimension $1$,  $\partial_0 W$ and $\partial_1 W$,
 and a $G$--equivariant map \[g\co (W,\partial_1 W)\rightarrow (X,A)\] such that $\partial W$ is $G$--diffeomorphic to 
$\partial_0 W \cup \partial_1 W$, $\partial_0 W$ is $G$--diffeomorphic to $M_1\amalg M_2$ with $g_{|\partial_0 W}=f_1\amalg f_2$ and $\partial\partial_0 W=\partial_0 W \cap\partial_1 W=\partial\partial_1 W$.
This gives an equivalence relation on singular $G$--manifolds over a pair of $G$--spaces $(X,A)$. By 
$\mathfrak{N}_n^G(X,A)$ we denote the set of such bordism classes. If $A$ is empty we shorten $\mathfrak{N}_n^G(X,\varnothing)$ to $\mathfrak{N}_n^G(X)$ and if $X$ is a point, we shorten $\mathfrak{N}^G_n(\pt)$ to $\mathfrak{N}^G_n.$
The direct sum over all dimensions $\mathfrak{N}_*^G(X,A)$ is called the \emph{real equivariant bordism module}. Addition is induced by taking the disjoint union on representative and the $\mathfrak{N}^G_n$--module structure is induced by taking the product of representatives. 
As in the non-equivariant case (compare \cite[Section 5]{CF79}), $\mathfrak{N}_*^G(-,-)$ gives a $\mathbb{Z}$--graded equivariant homology theory, with long exact sequence
\[
\xymatrix{\cdots\ar[r]&\mathfrak{N}_n^G(A)\ar[r]^-{i_*}&\mathfrak{N}_n^G(X)\ar[r]^-{j_*}&\mathfrak{N}_{n}^G(X,A)\ar[r]^-{\partial}&\mathfrak{N}_{n-1}^G(A)\ar[r]&\cdots}.
\]
\subsection{Homotopic real equivariant bordism}\label{esection}
In this section we restrict ourselves to a finite abelian group $G$ for simplicity. Many of the constructions can be carried out for arbitrary compact Lie groups. A \emph{real $G$--representation} is a finite-dimensional real inner product space $V$ together with a smooth action of $G$ on $V$ through linear isometries.
We denote the trivial $n$--dimensional representation, i.e.~$\mathbb{R}^n$ with trivial $G$--action, by $\underline{n}$. Let $|V|$ denote the dimension of a representation $V$ and $S^V$ its one-point compactification. 
Let $J$ be a complete set of representatives of equivalence classes of non-trivial irreducible representations of $G$. 
Notice that $|J|=|G|-1$ and that there is an isomorphism 
\[R\cong \underline{1}\oplus\bigoplus_{V\in J} V\] where $R$ denotes the regular representation. For the basic notions of representations of finite abelian groups see for example \cite[Sections 1.2 and 3.1]{S77}.
A complete $G$--universe in the sense of \cite[Chapter IX, Section 2]{M96} can be given by the countable direct sum
\[\mathcal{U}\colonequals R^{\oplus\infty}.\]
Finite $G$--subspaces of $\mathcal{U}$ will be called indexing spaces. By $BO^G(n)$ we denote the Grassmannian of $n$--dimensional subspaces of $\mathcal{U}$. 
For the trivial group $\{e\}$ a complete $\{e\}$--universe is given by $\mathbb{R}^{\oplus\infty}$ and $BO^{\{e\}}(n)$ is also denoted by $BO(n)$. 
Generally, $BO^G(n)$ is a classifying space for $n$--dimensional real $G$--vector bundles, compare \cite[Chapter I, Section 9]{tD87}. 
Let $\xi_n^G$ be the tautological bundle over $BO^G(n)$ and $T(\xi_n^G)$ be its Thom space. For two indexing spaces $V$ and $W$, such that $V\subset W$, we denote the orthogonal complement of $V$ in $W$ by $W-V$. 
Let $\pi$ be the product $G$--vector bundle with fiber $W-V$. 
The classification of the product bundle $\pi\times\xi_n^G$ yields a $G$--equivariant map 
\[\sigma_{V,W}\co \Sigma^{W-V}T(\xi_{|V|}^G)\to T(\xi_{|W|}^G)\] by passing to Thom spaces. 

For $n\geq 0$, the \emph{homotopic equivariant real bordism groups} are defined in non-negative degrees as
% without using the notion of equivariant spectra explicitly. Define 
\begin{align*}MO_n^G&\colonequals\colim\limits_{W\in \mathcal{U}}[S^{W\oplus \underline{n}},T(\xi_{|W|}^G)]^G=\colim\limits_{\substack{Z\in \mathcal{U}\\\underline{n}\subset Z}}[S^{Z},T(\xi_{|Z|-n}^G)]^G 
\intertext{and in negative degrees as}
		MO_{-n}^G&\colonequals\colim\limits_{W\in \mathcal{U}}[S^{W},T(\xi_{|W|+n}^G)]^G.
\end{align*}
Here $[-,-]^G$ denotes $G$--homotopy classes of $G$--equivariant maps, 
and the structure of an abelian group can be given to the colimit by using the group structures of homotopy classes of maps from $S^W$ for representations $W$ containing copies of the trivial representation. 
For $V\subset W$, the map in the colimit is obtained by smashing with $S^{W-V}$ and then 
composing with the map $\sigma_{V,W}$ (or with the map $\sigma_{V\oplus\underline{n},W\oplus\underline{n}}$ in case of negative degrees respectively).  

We can view an indexing space $V$ as a $|V|$--dimensional $G$--vector bundle over a point. 
Classifying this bundle and passing to Thom spaces gives a map $S^V\to T(\xi_{|V|})$. If we precompose with the inclusion $S^0\to S^V$, which is induced by the zero map $0\to V$, we obtain an element 
$\varepsilon_V\in MO_{-|V|}^G$, which is called the \emph{Euler class of $V$}.  

\subsection{Equivariant spectra and (co)homology theories}
A more conceptual approach to homotopic real equivariant  bordism is via equivariant spectra. 
We refer the reader to \cite[Chapters XII and XIII]{M96} and \cite[Chapters I, II and X]{LMS86} for the definition of $RO(G)$--graded equivariant (ring) (pre)spectra indexed on a complete $G$--universe and $RO(G)$--graded (co)homology theories  
and sketch the structures relevant to us. 

Thom spaces of appropriate Grassmannians and suspension maps described in the previous section constitute an equivariant prespectrum indexed on the universe $\mathcal{U}.$ 
The spectrification of this prespectrum is called \emph{real equivariant Thom spectrum} $MO^G$ and yields an $RO(G)$--graded homology theory associated to it, which is called \emph{homotopic real equivariant bordism}. 
It turns out that it's sufficient to consider integer gradings instead of $RO(G)$, in view of periodicity isomorphisms (see \cite[Chapter XV, Section 2, p. 157]{M96}).
This leads to the ad hoc definition of homotopic real equivariant bordism we have given above. 
\subsection{The equivariant Pontryagin\texorpdfstring{--}{-}Thom map}\label{injPT}
We give a short description of the Pontryagin--Thom map. 
 For every $k\in\mathbb{N}$ we construct a map $\PT\co \mathfrak{N}^G_k\rightarrow MO^G_k$ as follows. Given an element $[M]$ in $\mathfrak{N}^G_k$ represented by a $k$--dimensional $G$--manifold $M$, choose a $G$--representation $Z$, and an embedding of $M$ in $Z$.
(The fact that this is possible is the 	Mostow-Palais theorem, see \cite{M57} and \cite{P57}. A proof is also given by Wasserman \cite[\S1]{W69}). 
A tubular neighborhood $N$ of the image of $M$ in $Z$ is diffeomorphic to the total space $E(\nu)$ of the normal bundle $\nu$; compare \cite[Chapter 3, Section 22]{CF79}.
We define a map  \[t\co S^Z\rightarrow T\nu\] by sending $N$, viewed as a subset of $S^Z$, to $E(\nu)$ viewed as a subset of $T\nu$ via the diffeomorphism and send everything else, that is $S^Z-N$, 
to the base point of $T\nu.$ 
The classification of the normal bundle gives a map $Tf\co T\nu\rightarrow T(\xi^G_{|Z|-k})$ by passing to Thom spaces and hence we get a homotopy class $[Tf\circ t]\in [S^Z,T(\xi^G_{|Z|-k})]^G$, 
which represents an element in the colimit $MO^G_k$. This element is defined to be $\PT([M]).$                                                                                                                                                                                                                                                                                        
This  generalization of the classical Pontryagin--Thom construction is due to tom Dieck, \cite[\S1]{tD71}. Also compare \cite[\S3]{BH72} and \cite[p. 681]{H05}.
\begin{theorem}\label{PTnt}
 The above construction is well defined and a group homomorphism. It induces a ring homomorphism and a homomorphism of $\mathfrak{N}_*$--modules \[\PT\co \mathfrak{N}_*^G\rightarrow MO^G_*.\]
Furthermore the construction induces a natural transformation of $\mathbb{Z}$--graded equivariant homology theories $\PT\co \mathfrak{N}_*^G(-)\rightarrow MO_*^G(-).$
\end{theorem}
 One major ingredient for the proof of Theorem \ref{coro1} is the observation that the equivariant Pontryagin--Thom map is injective for certain groups. 
For a pair of $G$--spaces $(X,A)$ with ``good local properties'', i.e.~such that excision can be applied, the following has been shown by tom Dieck.
\begin{theorem}[see {\cite[Theorem 2]{tD71}}]\label{ptinj}
 For $G=\mathbb{Z}/2\times\cdots\times\mathbb{Z}/2$, the Pontryagin--Thom map $\PT\co \mathfrak{N}_*^G(X,A)\to MO_*^G(X,A)$ is a monomorphism. 
\end{theorem}
For non-trivial $G$ it can be seen that the Euler classes $\varepsilon_V$ for $G$--representations $V$ without trivial summand are non-trivial elements in $MO_{-|V|}^G$, see \cite[Chapter XV, Lemma 3.1]{M96}. \
Hence the Pontryagin--Thom map is not surjective, since it's a graded map and $\mathfrak{N}^G_*$ has no elements of negative degree by definition. 
\subsection{Families of subgroups}\label{CFtDes}
It's useful to consider geometric bordism groups with restricted iso\-tropy. We review concepts of \cite[Sections 5 and 6]{CF66} and \cite[Chapter XV, Section 3]{M96}.
All subgroups considered in this section are required to be closed.
A \emph{family of subgroups} $\mathcal{F}$ of $G$ is a set of subgroups of $G$ that is closed under conjugation and closed under taking subgroups.
 The family of all subgroups in $G$ is 
\begin{align*}
\mathcal{A}&\colonequals\{H\subset G\,|\,H \text{ closed subgroup in } G\} \\\intertext{and the family of all proper subgroups in $G$ is} 
\mathcal{P}&\colonequals\{H\subset G\,|\, H\neq G \text{ closed subgroup in } G\}.
\end{align*}
Let $\mathcal{F}'\subset\mathcal{F}$ be a pair of families of subgroups of $G$.  An \emph{$\mathcal{F}$--manifold} is defined to be a $G$--manifold $M$ such that all isotropy groups of $M$ are in $\mathcal{F}$ 
and an \emph{$(\mathcal{F},\mathcal{F}')$--manifold} is an $\mathcal{F}$--manifold such that $\partial M$ is an $\mathcal{F}'$--manifold.
With suitable bordisms we then define groups $\mathfrak{N}^G_*[\mathcal{F},\mathcal{F}']$ of bordism classes of $(\mathcal{F},\mathcal{F}')$--manifolds.
To define a similar concept in homotopic bordism, we consider certain types of universal spaces. 
Let $\mathcal{F}$ be a family of subgroups. There is a space $E\mathcal{F}$ called a \emph{universal $\mathcal{F}$--space of $G$}, which is unique up to $G$--homotopy, with the following properties: 
$(E\mathcal{F})^H$ is (non-equivariantly) contractible for $H\in\mathcal{F}$ and it's empty for $H\notin\mathcal{F}.$ Given an $\mathcal{F}$--manifold $M$, there is one and, up to homotopy, only one $G$--equivariant map $M\to E\mathcal{F}$. For more on this space see \cite[Satz 1]{tD72}, \cite[Chapter I, Theorem (6.6)]{tD87} and \cite[p.~45]{M96}.
With these universal $\mathcal{F}$--spaces we make the following identification:
\[\mathfrak{N}^G_*[\mathcal{F},\mathcal{F}']\cong \mathfrak{N}^G_*(E\mathcal{F},E\mathcal{F}').\]
This leads to the the definition
\[MO_*^G[\mathcal{F},\mathcal{F}']\colonequals MO_*^G(E\mathcal{F},E\mathcal{F}').\]
The long exact sequence of the pair $(E\mathcal{F},E\mathcal {F'})$ for $\mathfrak{N}_*^G(-,-)$ gives 
\[
\xymatrix{\cdots\ar[r] &\mathfrak{N}^G_*[\mathcal{F}']\ar[r]^-{i_{\mathfrak N}} &\mathfrak{N}^G_*[\mathcal{F}]\ar[r]^-{j_\mathfrak N} &\mathfrak{N}^G_*[\mathcal{F},\mathcal{F}']\ar[r]^-{\partial_\mathfrak N} &\mathfrak{N}^G_{*-1}[\mathcal{F}']\ar[r]&\cdots.}
\]
We call this the \emph{Conner--Floyd exact sequence} and it has a geometric interpretation: The map $\partial$ is actually induced by taking boundaries of singular $(\mathcal{F},\mathcal{F}')$--manifolds.
For real homotopical equivariant bordism $MO^G_*(-)$ we call the corresponding long exact sequence the \emph{tom Dieck exact sequence:}\[
\xymatrix{\cdots\ar[r] &MO^G_*[\mathcal{F}']\ar[r]^-{i_{MO}} &MO^G_*[\mathcal{F}]\ar[r]^-{j_{MO}} &MO^G_*[\mathcal{F},\mathcal{F}']\ar[r]^-{\partial_{MO}} &MO^G_{*-1}[\mathcal{F}']\ar[r]&\cdots}.
\]

\section{The case \texorpdfstring{$G=\mathbb{Z}/2\times\cdots\times\mathbb{Z}/2$}{G=Z/2...Z/2}}\label{drei}
From now on we set $G=(\mathbb{Z}/2)^{l}$ for some $l>0$ and fix a complete set $J$ of representatives of isomorphism classes of non-trivial irreducible $G$--representations.  
\subsection{Map from homotopic bordism} 
First we define certain indexing tools. 
The free abelian group $\mathbb{Z}J$ can be considered to be an additive subgroup of the real representation ring $RO(G)$.
We set  \[AO_*(G)\colonequals\mathbb{Z}[\mathbb{Z}J].\]
This is a graded ring; the grading is induced by the virtual dimension of elements in $\mathbb{Z}J\subset RO(G)$. 
  We have an isomorphism \[AO_*(G)\cong\mathbb{Z}{[e_V,e^{-1}_V]}_{V\in J},\] for indeterminates $e_V$ of degree $-|V|$ and $e_V^{-1}$ of degree $|V|$ with the obvious relations. It is induced by 
 \[\mathbb{Z}J\ni\sum_{V\in J}\alpha_VV\mapsto\prod_{V\in J}e_V^{-\alpha_V}.\]
 Compare the analogous complex definitions in \cite[p.683]{H05}.

Let's consider the fixed set of the equivariant Thom space.
\begin{lemma}[compare {\cite[Proposition 4.7]{S01}} and {\cite[Lemma 2.1]{tD70}}]\label{thomfix} We have the following homotopy equivalence:
\[(T(\xi^G_n))^G\simeq \bigvee_{\substack{W\in RO^+(G)\\ |W|=n}} T(\xi_{|W^G|})\sma \left(\prod_{V\in J}BO(\nu_V(W))\right)_+,\] 
where $RO^+(G)$ is a set of $G$--representations, containing one from every isomorphism class. It can be viewed as a subset of $RO(G)$. The number of times $V$ appears as a direct summand of $W$ is denoted by $\nu_V(W).$
\end{lemma}
\begin{proof}The space $(BO^G(n))^G$ classifies $n$--dimensional $G$--vector bundles $E$ over a base space $X$ with trivial $G$--action. Such an $E$ decomposes into $V$--isotypical subbundles as follows:
\[E\cong \bigoplus_{V\in J\cup\{\underline{1}\}}E_V\otimes_\mathbb{R}V.\]
This is a well-known result; Segal gives a proof in \cite[Proposition 2]{Se68}, and for $G$ a finite group Oliver \cite[Appendix]{O96} also gives a proof.
The base $X$ decomposes into a disjoint union of subspaces $X_W$ over which $E$ has constant fiber $W$. The restriction of $E$ to $X_W$ is classified by a map to \[\prod_{V\in J\cup\{1\}}BO(\nu_V(W))=BO(|W^G|)\times\prod_{V\in J}BO(\nu_V(W)),\] 
where the map to the factor $BO(\nu_V(W))$ is a classifying map of $E_V$. The universal bundle over this space is the product $\xi_{|W^G|}\times \xi$, 
where $\xi_{|W^G|}$ is the $|W^G|$--dimensional universal bundle and $\xi$ is a $G$--vector bundle without trivial direct summand in the fiber, so $E(\xi)^G$ is the zero section $\prod_{V\in J}BO(\nu_V(W))$. 
Taking all these classifying spaces together we get 
\begin{align*}(BO^G(n))^G&\simeq \coprod_{W\in RO^+(G)}\left(BO(|W^G|)\times\prod_{V\in J}BO(\nu_V(W)) \right)\\
\intertext{and 
passing to Thom spaces gives} 
(T(\xi^G_n))^G&
\simeq \bigvee_{\substack{W\in RO^+(G)\\ |W|=n}} T(\xi_{|W^G|})\sma \left(\prod_{V\in J}BO(\nu_V(W))\right)_+.
\end{align*}
\end{proof}
We want to define a ring homomorphism $\phi_{MO}\co MO^G_*\to MO_*[e_V, e_V^{-1},Y_{V,d}]$ and start with a map which is essentially restriction to fixed sets. 
To be more precise, we pass from the Thom spectrum $MO^G$ to its geometric fixed sets spectrum $\Phi^GMO^G$.
For detailed definitions see \cite[Chapter XVI, Section 3]{M96} (also compare \cite[Chapter II \S 9 and Chapter I \S 3]{LMS86} and \cite[Definition 4.3]{S01}).

We consider the classifying space $BO(n)$ and define
\[BO\colonequals \colim_nBO(n).\]
The space $BO$ carries an $H$--space structure that arises from classifying the product of bundles. 
We set \[B\colonequals BO^{\times |J|}.\] Compare \cite[\S2]{tD70} for the complex analogue $BU$; 
notice that since $G$ is finite, we don't need to consider a proper subset of $BO^{\times |J|}.$
\begin{propo}[compare {\cite[Theorem 4.9]{S01}}]\label{fixequi}
 There is an equivalence of ring spectra \[\Phi^GMO^G\simeq I_{RO(G)}\sma MO\sma B_+\] with \[I_{RO(G)}\colonequals\bigvee_{\substack{W\in RO(G)\\|W|=0}} S^{|W^G|}.\]
Here $S^{|W^G|}$ denotes a suspended sphere spectrum and $|W^G|$ denotes the (possibly negative) virtual dimension of $W'^G-W''^G$ if $W$ is represented by $W'-W''$.
 %we view $W$ as an element in $\mathbb{Z}(J\cup\{\underline{1}\}$ such that
% $W=\sum_{V\in J\cup\{\underline{1}\}}\alpha_V V$, let $|W^G|$ be $\alpha_{\underline{1}}\in\mathbb{Z}.$
Note that $I_{RO(G)}$ carries the structure of a ring spectrum induced by the isomorphism 
\[S^{|W^G|}\sma S^{|V^G|}\to S^{|(W\oplus V)^G|}\]
for elements $W,V\in RO(G)$, $|W|=|V|=0.$
\end{propo}

The proof of the analogous complex case \cite[Theorem 4.9]{S01} can be translated word-by-word; the essential ingredient is the description of the equivariant Thom space of Lemma \ref{thomfix}.  

\begin{propo}[compare {\cite[p.\ 684]{H05}}]\label{MOBiso}
 There is an isomorphism of graded rings 
\[\xymatrix{(I_{RO(G)}\sma MO\sma B_+)_*\ar[r]^\cong&MO_*(B)\otimes AO_*(G)}.\]
\end{propo}
\begin{proof}
The  spectrum $I_{RO(G)}\sma MO\sma B_+$ can be viewed as a wedge of suspended copies of $MO\sma B_+$. For such a copy indexed by an element $W-U\in RO(G)$ of virtual dimension zero with $W=W^G\oplus (W^G)^{\perp}$ and $U=U^G\oplus (U^G)^\perp$ we identify $(S^{(W-U)^G}\sma MO\sma B_+)_*$ with \[MO_*(B)\otimes(e_{(W^G)^\perp}\cdot e^{-1}_{(U^G)^\perp})\subset MO_*(B)\otimes AO_*(G).\]
This induces the desired isomorphism. \end{proof}
\begin{propo}[compare {\cite[Theorem 4.10]{S01}}]\label{kochman}
 There is an isomorphism of graded rings
\[MO_*(B)\otimes AO_*(G)\cong MO_*[e_V,e_V^{-1},Y_{d,V}],\] 
where $V$ runs through $J$ and $d>1$. Here the degree of the $e_V$'s is $1$, the degree of the $e_V^{-1}$'s is $-1$ and the degree of the $Y_{d,V}$'s is $d$.  
\end{propo}
\begin{proof}Since $J$ is finite we have the isomorphism of the Künneth formula \[MO_*(\prod_{V\in J}BO)\cong\bigotimes_{MO_*}^{V\in J}MO_*(BO).\]
Conner and Floyd used the Atiyah--Hirzebruch spectral sequence to calculate $MO_*(-)$, see \cite[Theorem 8.3, Theorem 17.1]{CF79}. For $MO_*(BO)$ we obtain
\begin{align*}MO_*(BO)&\cong MO_*{[X_i]}_{1\leq i},\end{align*}
 where each generator $X_i$ has degree $i$, (compare \cite[Propositions 2.3.7 and 2.4.3]{K96}). 
%The generators $X_i$ can be represented by a map $\mathbb{R}P^i\to BO(1)\to BO$ classifying the tautological line bundle $E_i\to \mathbb{R}P^i$,
With the identification $AO_*(G)\cong\mathbb{Z}{[e_V,e^{-1}_V]}_{V\in J}$ we get an isomorphism
\begin{align*}
\begin{split}
MO_*(B)\otimes AO_*(G)& \cong MO_*{[e_V,e_V^{-1},X_{i,V}]}_{V\in J, 1\leq i} \\
&\cong MO_*{[e_V,e_V^{-1},X_{d-|V|,V}\cdot e_V^{-1}]}_{V\in J, 1+|V|\leq d}. \label{polyiso} 
\end{split}\tag{$\star$} 
\end{align*}

\begin{definition}
 For a $G$--representation $V\in J$ and $1+|V|\leq d$ we set \[Y_{d,V}\colonequals X_{d-|V|,V}\cdot e_V^{-1}.\] \end{definition}
We identify $Y_{d,V}$ with the image of \[X_{d-|V|}\otimes e^{-1}_V\in MO_{d-|V|}(BO)\otimes AO_{|V|}(G) \]
under the inclusion of $BO$ as $V$--th factor in $B$ viewed, via the isomorphism \eqref{polyiso}, as an element in $MO_*[e_V,e_V^{-1},X_{i,V}].$
 Notice that $Y_{d,V}$ is defined in such a way that its dimension is $d$, and that all representations in $J$ have dimension one. 
With this definition we get the desired isomorphism\[MO_*(B)\otimes AO_*(G)\cong MO_*{[e_V,e_V^{-1},Y_{d,V}]}_{V\in J, d>1}.\] 
\end{proof} 

\begin{definition}\label{phidef}
 We combine the results of Propositions \ref{fixequi} to \ref{kochman} to define a map
\[\phi_{MO}\co MO^G_*\to MO_*[e_V,e^{-1}_V,Y_{d,V}]\] as follows:
\[\xymatrix{MO_*^G\ar[rr]^-{\text{restriction}}_-{\text{to fixed sets}}&&\Phi^GMO^G_*\ar[rr]^-{\cong}_-{\text{Proposition \ref{fixequi}}}&&(I_{RO(G)}\sma MO\sma B_+)_*\\
\ar[rr]^-{\cong}_-{\text{Proposition \ref{MOBiso}}}&& MO_*(B)\otimes AO_*(G)\ar[rr]^{\cong}_-{\text{Proposition \ref{kochman}}} &&MO_*[e_V,e_V^{-1},Y_{d,V}]}\] where the first map is given by restriction to fixed points. Not including the last isomorphism we get a map 
\[\tilde\phi_{MO}\co MO^G_*\to MO_*(B)\otimes AO_*(G).\]
\end{definition}

\subsection{Localization}\label{fixinj}
The reason for naming the indeterminates $e_V$ is explained in the lemma below, which can be proved by chasing through the definition of the map $\phi_{MO}$ and the Euler classes $\varepsilon_V$.
\begin{lemma}[compare {\cite[p.\ 685]{H05}}]\label{evphiev}
 For a non-trivial irreducible representation $V$ and the corresponding Euler class $\varepsilon_V\in MO^G_*$ we have
\[\phi_{MO}(\varepsilon_V)=e_V\in MO_*[e_V,e_V^{-1},Y_{d,V}].\]
Notice that $\phi_{MO}(\varepsilon_V)$ is invertible in $MO_*[e_V,e_V^{-1},Y_{d,V}]$.
\end{lemma}

We give an alternative description of $\phi_{MO}$.
The key steps are the following two results by tom Dieck. Also compare the complex version of the following proposition \cite[Corollary 5.2]{S01}.
\begin{propo}[see {\cite[Theorem 1(b)]{tD71}}]
Let $S$ be the set of Euler classes of non-trivial irreducible representations in $MO_*^G$. Then the localization map into the ring of quotients
$\lambda\co MO_*^G\to S^{-1}MO_*^G$ is injective. 
\end{propo}
The map $\tilde{\phi}_{MO}$
sends all elements of $S$ to units, since the image $\phi_{MO}(\varepsilon_V)=e_V$ of an Euler class $\varepsilon_V$ is a unit in $MO_*[e_V,e_V^{-1},Y_{d,V}]$ by Lemma \ref{evphiev}. 
Hence the universal property of localization gives rise to a unique map
\[\tilde{\Phi}_{MO}\co S^{-1}MO_*^G\to MO_*(B)\otimes AO_*(B)\] such that $\tilde{\Phi}_{MO}\circ\lambda=\tilde{\phi}_{MO}.$  We cite the following result without proof.
 \begin{propo}[see {\cite[p.\ 217]{tD71}} and {\cite[Hilfssatz 2]{td70a}}]
  The map \[\tilde{\Phi}_{MO}\co  S^{-1}MO_*^G\to MO_*(B)\otimes AO_*(B)\] is an isomorphism. 
 \end{propo}
Complex versions of the propositions are \cite[Theorem 3.1]{tD70} and \cite[Corollary 4.15]{S01}. For $G=\mathbb{Z}/2$ the corresponding statement is \cite[Corollary 3.19]{S02}.
Fitting it all together and composing with the isomorphism of Proposition \ref{kochman} gives the following commutative diagram:
\[\xymatrix{MO_*^G\ar[r]^-{\lambda}\ar[rd]^-{\tilde{\phi}_{MO}}\ar@/_3.25pc/[rdd]_-{\phi_{MO}} & S^{-1}MO_*^G\ar[d]^-{\tilde{\Phi}_{MO}}\ar@/^6pc/[dd]^-{\Phi_{MO}}\\
	&MO_*(B)\otimes AO_*(G)\ar[d]_-{\text{Proposition \ref{kochman}}}^\cong\\
      &MO_*[e_V,e_V^{-1},Y_{d,V}]}\]
\begin{corollary}
 We have $\Phi_{MO}\circ\lambda=\phi_{MO}$ and $\phi_{MO}$ is a monomorphism. 
\end{corollary}
\subsection{Map from geometric bordism}\label{geomap}
Next we want to construct a map \[\phi_{\mathfrak{N}}\co \mathfrak{N}^G_*\to MO_*[e_V^{-1},Y_{d,V}].\]
See \cite[Proposition 3]{H05} for the analogous construction in the complex case. 
Let $M^n$ be a manifold representing an element $[M]\in\mathfrak{N}_*^G$ and let $F\subset M^G$ be a connected component of the fixed set  of $M$. Then $F$ is embedded in $M$. 
The normal bundle $\nu_F^M$ of $F$ in $M$ is a real $G$--vector bundle of dimension $m$ and only the zero vector is fixed by the $G$--action on each fiber. This bundle decomposes as follows:
\[\nu_F^M=\bigoplus_{k=1}^{|J|}E_k\otimes_\mathbb{R} V_k\] for real vector bundles $E_k$ and irreducible $G$--representations $V_k.$ Notice that for the groups we consider $|J|=2^l-1$.
Define \[b_F\colonequals\overline{b}_F\otimes(e^{-|E_1|}_{V_1}\cdot \dots \cdot e^{-|E_{|J|}|}_{V_{|J|}})\in MO_{n-m}(B)\otimes AO_m(G),\] where $\overline{b}_F\in MO_{n-m}(B)$ is represented by a map $F\to B$ with $V_k$--th component the classifying map for $E_k.$ 
Altogether we get  the following map:
\begin{align*}\tilde\phi_\mathfrak N&\co \mathfrak N^G_*\to MO_*(B)\otimes AO_*(G)\\
 [M]&\mapsto\sum_{F\subset M^G}b_F\in (M(B)\otimes AO(G))_n.
\end{align*}
Compare tom Dieck's description of the map in \cite[Section 5]{tD71}.
Composing with the isomorphism of Proposition \ref{kochman} we get a map $\phi_\mathfrak N\co \mathfrak{N}^G_*\to MO_*[e_V,e_V^{-1},Y_{d,V}]$ and, as we will show next, its image is contained in  $MO_*[e_V^{-1},Y_{d,V}].$
  The element $\overline{b}_F\in MO_{n-k}(B)$ is represented by a map 
$F\to BO(|E_1|)\times \cdots \times BO(|E_{|J|}|),$ so $\overline{b}_F$ lies in
\begin{align*}
MO_*( \prod_{k=1}^{|J|} BO(|E_k|))&\cong\bigotimes_{MO_*}^{1\leq k\leq |J|} MO_*(BO(|E_k|))\\
 &\cong\bigotimes_{MO_*}^{1\leq k\leq |J|}MO_*[X_{1},\dots,X_{|E_k|}]&\subset\bigotimes^{1\leq k\leq |J|}_{MO_*} MO_*{[X_{d,V_k}]}_{d>0}.
\end{align*}
In fact every element in $MO_*(BO(|E_k|))$ can be written as a sum of monomials with at most $|E_k|$ factors $X_{d,V_j}.$
(Compare the classical calculations in \cite[Theorem 8.3]{CF79} and \cite[Propositions 2.4.3 and 2.3.7]{K96}.)
By definition of the $Y_{d,V}$'s we have $X_{d,V}=Y_{d+|V|, V}\cdot e_V$ and this asserts that $e_{V_k}$ appears at most $|E_k|$ times as factor in $\overline{b}_F$ and hence appears in non-negative degree (i.e.~with non-positive exponent) in $\overline{b}_F$. It follows that the sum $b_F$ lies in $MO_*[e_V^{-1}, Y_{d,V}]$ and hence so does $\phi_\mathfrak{N}([M])$.

Next we show how $\phi_{MO}$ corresponds to $\phi_\mathfrak{N}$. 
\begin{definition}[compare {\cite[p.\ 354]{tD70}}]\label{nuiota}
The inverse of the $H$--space $B$ structure gives a map 
$-^{-1}\co B\to B.$
This induces a map $\nu\co  MO_*(B)\to MO_*(B)$, which has order 2. 
Together with the isomorphism of Proposition \ref{kochman},  $\nu$ induces a map $\iota\co MO_*[e_V,e_V^{-1},Y_{d,V}]\to MO_*[e_V,e_V^{-1},Y_{d,V}]$, such that the following diagram commutes:
\[\xymatrix{MO_*(B)\otimes AO(G)\ar[d]^-{\nu\otimes \id}\ar[r]^-\cong&MO_*[e_V,e_V^{-1},Y_{d,V}]\ar[d]^-\iota\\MO_*(B)\otimes AO(G)\ar[r]^-{\cong}&MO_*[e_V,e_V^{-1},Y_{d,V}].}\]
\end{definition}
\begin{remark}
 Notice for the complex analogue of our Lemma \ref{evphiev}, namely the statement
\[\iota\circ\phi_{MU}\circ\Psi([P(\mathbb{C}^d\oplus V)])=Y_{V,d}+e_{V^*}^{-d},\]
in the notation used there \cite[p.\ 685]{H05}, an analogous map $\iota$ is used. However, since 
\[(\nu\otimes \id)(1\otimes e_V)=\nu(1)\otimes e_V=1\otimes e_V,\] we have
 $\iota\circ\phi_{MO}(\varepsilon_V)=e_V=\phi_{MO}(\varepsilon_V).$ (In general $\iota\circ\phi_{MO}\neq\phi_{MO}.$)
\end{remark}
\begin{lemma}\label{phicommu}
The following diagram commutes:
\[\xymatrix{\mathfrak{N}^G_*\ar[r]^-{\tilde{\phi}_\mathfrak{N}}\ar[d]^-{\PT}& MO_*(B)\otimes AO_*(G)\\MO^G_*\ar[r]^-{\tilde{\phi}_{MO}}&MO_*(B)\otimes AO_*(G)\ar[u]_-{\nu\otimes \id}.}\]
\end{lemma}
A detailed proof of the analogous complex result that translates immediately to the real case is given by tom Dieck \cite[Proposition 4.1]{tD70}. 

\subsection{Bordism with respect to families of subgroups}
\begin{propo}[compare {\cite[Proposition 4]{H05}}]\label{kappa1}
 There is an isomorphism \[\kappa_\mathfrak{N}\co \mathfrak{N}_*^G[\mathcal{A},\mathcal{P}]\to MO_*[e^{-1}_V,Y_{d,V}]\] 
such that the following diagram commutes:
\[\xymatrix{&&&  MO_*[e_V^{-1},Y_{d,V}]&\\
 \mathfrak{N}_*^G[\mathcal{A}]\ar[rr]_{j_\mathfrak N}\ar[rrru]^-{\phi_\mathfrak{N}}&& \mathfrak{N}_*^G[\mathcal{A},\mathcal{P}]\ar[ru]_-{\kappa_\mathfrak{N}} &. }\]
The map $j_{\mathfrak{N}}$ comes from the Conner--Floyd exact sequence (see Section \ref{CFtDes}).
\end{propo}
\begin{proof}
The isomorphism $\kappa_\mathfrak{N}$ is constructed as follows. A manifold  $N$ that represents an element $[N]\in \mathfrak{N}^G_n[\mathcal{A},\mathcal{P}]$ is bordant to every closed tubular neighborhood of its fixed set 
$M\colonequals N^G$, which lies in the interior of $N$, since there are no fixed points on the boundary. This can be seen by a straithening-the-angle argument and then giving the bordism explicitly as explained in \cite[Lemma (5.2)]{CF66}. From here one proceeds exactly as in the definition of the map $\phi_{\mathfrak{N}}$ (see Section \ref{geomap}), which ensures the commutativity of the diagram immediately. 
 To see that $\kappa_\mathfrak{N}$ is an isomorphism we give an inverse 
\[\kappa_\mathfrak{N}^{-1}\co MO_*[e_V^{-1},Y_{d,V}]\to\mathfrak{N}_*^G[\mathcal{A},\mathcal{P}].\]
The element $e^{-1}_V$ is sent to the class of the disc bundle of $V$ viewed as a bundle over a point. Since $V$ does not contain the trivial representation its unit disc bundle $D(V)$  has boundary $S(V)$ without fixed points. Then $\kappa_\mathfrak{N}$ sends this bundle back to $e_V^{-1}$, since we have the decomposition $\mathbb{R}\otimes_\mathbb{R}V\to *$ and the class of the map $*\to BO(1)$ classifying $\mathbb{R}$ gives $1\in MO_*(B)$,  so 
\[\kappa_\mathfrak{N}([V\to *])=1\otimes e_V^{-|\mathbb{R}|}=e_V^{-1}.\]
%TODOUsing explicit generators for the $X_{d,V}$ we define $\kappa_\mathfrak{N}^{-1}(Y_{d,V}).$
On $Y_{d,V}$ the inverse $\kappa_{\mathfrak N}^{-1}$ is constructed as follows:
Let $E_{d-|V|}$ denote the line bundle representing the generator $X_{d-|V|}$ (compare the proof of Proposition \ref{kochman}). Then $\kappa_\mathfrak N^{-1}(Y_{d,V})$ is defined to be the class of the disc bundle of $E_{d-|V|}\otimes V.$
As above we get
\[\kappa_\mathfrak N([E_{d-|V|}\otimes V])=X_{d-|V|}\otimes e_V^{|E_{d-|V|}|}=X_{d-|V|}\otimes e_V^{-1}=Y_{d,V}.\]
Now $\kappa_\mathfrak N^{-1}$ is defined by requiring it to be a homomorphism of $\mathfrak N_*$--modules. Clearly $\kappa_\mathfrak N^{-1}$ is a right and a left inverse of $\kappa_\mathfrak N.$
\end{proof}
\begin{propo}[compare {\cite[Proposition 4]{H05}}]\label{kappa2}
 There is an isomorphism \[\kappa_{MO}\co MO_*^G[\mathcal{A},\mathcal{P}]\to MO_*[e_V,e^{-1}_V,Y_{d,V}]\] 
such that the following diagram commutes:
\[\xymatrix{&&&  MO_*[e_V, e_V^{-1},Y_{d,V}]&\\
 MO^G_*[\mathcal{A}]\ar[rr]_-{j_{MO}}\ar[rrru]^-{\phi_{MO}}&& MO_*^G[\mathcal{A},\mathcal{P}]\ar[ru]_-{\kappa_{MO}} &. }\]
The map $j_{MO}$ comes from the tom Dieck exact sequence (see Section \ref{CFtDes}).
\end{propo}
\begin{proof}
 By definition $MO_*^G[\mathcal{A},\mathcal{P}]$ is $MO_*^G(E\mathcal{A},E\mathcal{P})$. Let $\overline{\Sigma}E\mathcal{P}$ be the unreduced suspension of $E\mathcal{P}.$
%\[\overline{\Sigma}E\mathcal{P}\colonequals([0,1]\times E\mathcal{P})/\substack{(0,x)\sim(0,y)\\(1,x)\sim(1,y)}.\]
We identify $E\mathcal{P}$ with $\frac{1}{2}\times E\mathcal{P}\subset \overline{\Sigma}E\mathcal{P}$ and denote the upper cone by $C^+E\mathcal{P}\colonequals[\frac{1}{3},1]\times E\mathcal{P}\subset \overline{\Sigma}E\mathcal{P}$ and the lower cone by 
$C^-E\mathcal{P}\colonequals[0,\frac{2}{3}]\times E\mathcal{P}.$ Then $(E\mathcal{A},E\mathcal{P})\simeq(C^-E\mathcal{P},E\mathcal{P})$ and the inclusion
$(C^-E\mathcal{P},E\mathcal{P})\to (\overline{\Sigma}E\mathcal{P},C^+E\mathcal{P})$ gives an isomorphism via excision:
\[MO^G_*(E\mathcal{A},E\mathcal{P})\cong MO_*^G(\overline{\Sigma}E\mathcal{P},C^+E\mathcal{P}). \]
To calculate $MO_*^G(\overline{\Sigma}E\mathcal{P},C^+E\mathcal{P})=MO_*^G(\overline{\Sigma}E\mathcal{P})$
we apply Lemma 4.2 of \cite{S01}:
\begin{lemma} Let $Z$ be a $G$--complex such that $Z^G\simeq S^0$ and $Z^H$ is contractible for any proper subgroup $H\subsetneq G$. For a finite $G$--complex $X$ the restriction map
 \[(\Map(X,Y\sma Z))^G\to \Map(X^G,(Y\sma Z)^G)=\Map(X^G,Y^G) \]
is a homotopy equivalence. 
\end{lemma} 
\noindent For the $G$--complex $\overline{\Sigma}E\mathcal{P}$ (compare Section \ref{CFtDes}) and any proper subgroup $H\subsetneq G$, the space $(\overline{\Sigma}E\mathcal{P})^H$ is contractible by the construction of $E\mathcal{P}$ and furthermore \[(\overline{\Sigma}E\mathcal{P})^G\simeq S^0.\] 
Since $S^W$ is a finite $G$--complex, we obtain
\begin{align*}
MO_n^G(\overline{\Sigma}E\mathcal{P})&\cong\colim_W[S^W,T(\xi_{|W|+n}^G)\sma\overline{\Sigma}E\mathcal{P}]^G\\
    &\cong \colim_W [(S^W)^G, (T(\xi_{|W|+n}^G))^G]\\
    &\cong \Phi^GMO_n^G.
\end{align*}
Combining this with the isomorphism $\Phi^GMO_n^G\cong MO_*[e_V, e_V^{-1},Y_{d,V}]$ of Propositions \ref{fixequi} to \ref{kochman}, we get the desired isomorphism $\kappa_{MO}$. 
To show commutativity of the following diagram, we look at the definitions of $\phi_{MO}$ and $\kappa_{MO}$:
\[\xymatrix{MO_n^G(E\mathcal{A})=MO^G_*[\mathcal{A}]=MO^G_*\ar[rr]\ar[d]^-j && \Phi^GMO^G_n\\
	  MO^G_*(E\mathcal{A},E\mathcal{P})\ar[rr]^-{\cong}&&MO^G_*(\overline{\Sigma}E\mathcal{P}).\ar[u]^{\cong}}
\]
Here the upper horizontal map is the first map in the definition of $\phi_{MO}$; it's restriction to fixed sets. The vertical map $j$ on the left hand side comes from the tom Dieck exact sequence. Let $f$ represent an element in $MO^G_n=MO_n^G(E\mathcal{A}),$
\[f\co S^W\to TO_n(\xi_{n+|W|}^G)\sma E\mathcal{A}.\]
Restricting to fixed sets gives an element represented by 
\[f^G\co (S^W)^G\to (TO_n(\xi_{n+|W|}^G))^G.\]
On the other hand we see that
\begin{align*}j(f)\in[S^W, T(\xi^G_{n+|W|})\sma E\mathcal{A}/E\mathcal{P}]^G&=[(S^W)^G, (T(\xi^G_{n+|W|})\sma (\overline{\Sigma}E\mathcal{P}))^G]\\&=[(S^W)^G,(T(\xi_{n+|W|}^G))^G]\end{align*} 
gives the same element and the diagram commutes. Combining this with isomorphisms of Propositions \ref{fixequi} to \ref{kochman} we get $\kappa_{MO}\circ j=\phi_{MO}.$
\end{proof}
\begin{propo}[compare {\cite[Proposition 4]{H05}}]\label{kappa3}
 The following diagram commutes:
\[\xymatrix{\mathfrak{N}^G_*[\mathcal{A},\mathcal{P}]\ar[d]^-{\PT[\mathcal{A},\mathcal{P}]}\ar[rr]^-{\kappa_\mathfrak{N}}&&MO_*[e_V^{-1},Y_{d,V}]\ar[d]^i\\
	     MO^G_*[\mathcal{A},\mathcal{P}]\ar[rr]^-{\iota\circ\kappa_{MO}}&&MO_*[e_V,e_V^{-1},Y_{d,V}]}\]
\end{propo}
This is essentially proved in the same way as Lemma \ref{phicommu}. Notice that $\iota$ corresponds to $\nu$ there (see Definition \ref{nuiota}). 
Given an element $[N]$ in $\mathfrak{N}_n^G[\mathcal{A},\mathcal{P}]$ we construct the element $i\circ\kappa_{\mathfrak{N}}([N])$ using the same notation as in the definition of $\kappa_\mathfrak{N}.$ Then 
\[i\circ\kappa_{\mathfrak{N}}([N])=\overline{b}_M\otimes e_{V_1}^{-|E_1|}\cdots e_{V_j}^{-|E_j|}\in MO_{n-k}\otimes AO_k(G).\]
On the other hand we choose an embedding $N\to U$ into a $G$--representation \[U=U^G\oplus \bigoplus_{i\in I}V_i\] and $\PT[\mathcal A,\mathcal P]([N])$ is then represented by a map
$S^U\to T(\xi^G_{|U|-n})\sma E\mathcal{P}$ classifying the normal bundle of the embedding. Considering the map 
\[S^{U^G}\to T(\xi^G_{|U|-k)})^G\] viewed as an element in $\phi^GMO^G$ we obtain the element $\kappa_{MO}\circ \PT[\mathcal A,\mathcal P]([N]).$ We examine the normal bundles of the embeddings
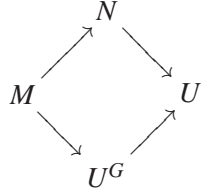
\[\xymatrix @ur {M\ar[r]\ar[d]&N\ar[d]\\U^G\ar[r]&U}\] 
and get the desired conclusion 
$\iota\circ\kappa_{MO}\circ \PT[\mathcal A,\mathcal P]([N])=i\circ \kappa_\mathfrak N([N]).$
\subsection{Statement and proof of main result}
\begin{theorem}[compare {\cite[Theorem 1]{H05}}]\label{coro1}\label{thee} The following diagram commutes and is a pull-back with all maps injective:
\[\xymatrix{\mathfrak N_*^G\ar[d]^-{\PT}\ar[rr]^-{\phi_\mathfrak N}&&MO_*[e^{-1}_V,Y_{d,V}]\ar[d]^-{i}\\
                 MO_*^G\ar[rr]^-{\iota\circ\phi_{MO}} && MO_*[e_V,e^{-1}_V,Y_{d,V}]. } \]
\end{theorem}
\begin{proof}
Putting together the exact sequence of the the pair of families of subgroups $(\mathcal{A},\mathcal{P})$ (see Section \ref{CFtDes}), the natural transformation coming from the Pontryagin--Thom map (see Theorem \ref{PTnt}) 
and commutative diagrams of Propositions \ref{kappa1}, \ref{kappa3} and \ref{kappa2}, we obtain the following commutative diagram with exact horizontal rows:
\[
\xymatrix{&&&  MO_n[e_V^{-1},Y_{d,V}]\ar'[d][dd]^(.4){i}|!{[ld];[rd]}\hole &\\
 \mathfrak{N}^G_n[\mathcal{A}]\ar[dd]^-{\PT[\mathcal{A}]}\ar[rr]_-{j_{\mathfrak N}}\ar[rrru]^-{\phi_\mathfrak{N}}&& \mathfrak{N}_n^G[\mathcal{A},\mathcal{P}]\ar[dd]^-{\PT[\mathcal{A},\mathcal{P}]}\ar[ru]^(.4){\cong}_-{\kappa_\mathfrak{N}}\ar[rr]^(.65){\partial_\mathfrak{N}} &&\mathfrak{N}_{n-1}^G[\mathcal{P}]\ar[dd]^-{\PT[\mathcal{P}]}\\
&&&  MO_n[e_V, e_V^{-1},Y_{d,V}]&\\
 MO^G_n[\mathcal{A}]\ar[rr]_-{j_{MO}}\ar[rrru]^-{\iota\circ\phi_{MO}} |!{[uurr];[rr]}\hole && MO_n^G[\mathcal{A},\mathcal{P}]\ar[ru]^(.4){\cong}_-{\iota\circ\kappa_{MO}}\ar[rr]^-{\partial_{MO}} &&MO_{n-1}^G[\mathcal{P}]. }
\]
Using the isomorphisms $\kappa_\mathfrak N$ and $\iota\circ\kappa_{MO}$ to substitute ${\PT[\mathcal{A},\mathcal{P}]}$ in the middle by the inclusion $i$, gives the following commutative diagram with short exact sequences as rows: 
\[\xymatrix{0\ar[r] &\mathfrak{N}^G_n\ar[d]^-{\PT}\ar[r]^-{\phi_\mathfrak{N}} &MO_n[e_V^{-1},Y_{d,V}]\ar[d]^-{i}\ar[rr]^-{\partial_\mathfrak{N}\circ\kappa^{-1}_\mathfrak{N}} &&\mathfrak{N}^G_{n-1}[\mathcal{P}]\ar[d]^-{\PT[\mathcal{P}]}\ar[r]&0\\
0\ar[r] &MO^G_n\ar[r]^-{\iota\circ\phi_{MO}} &MO^G_n[e_V,e_V^{-1},Y_{d,V}]\ar[rr]^-{\partial_{MO}\circ(\iota\circ\kappa_{MO})^{-1}} &&MO^G_{n-1}[\mathcal{P}]\ar[r]&0.}\]
The Pontryagin--Thom maps $\PT=\PT[\mathcal{A}]$ and $\PT[\mathcal{P}]$ are injective by  Section \ref{injPT} and so is the inclusion in the middle.
From that and the injectivity of $\iota\circ\phi_{MO}$ (see Section \ref{fixinj}) the injectivity of $\phi_\mathfrak{N}$ follows. 
(The injectivity of $\phi_{\mathfrak N}$ can also be deduced from the injectivity of $j_\mathfrak N$; see Proposition \ref{toralmono}). To prove the pullback property it suffices to show that an element $x\in \im i\cap\im \iota\circ\phi_{MO}$ comes from an element in $\mathfrak{N}_n^G$, which is done by a diagram chase.
\end{proof}
We identify $MO_*[e_V^{-1},Y_{d,V}]$ as a subring of $MO_*[e_V,e_V^{-1},Y_{d,V}]$ via $i$. 
\begin{corollary}[compare {\cite[Corollary 1]{H05}}]
 The following isomorphism of $MO_*$--algebras describes geometric equivariant bordism for $G=\mathbb{Z}/2\times\cdots\times\mathbb{Z}/2$: 
\[\mathfrak N_*^G\cong \iota\circ\phi_{MO}(MO_*^G)\cap MO_*[e_V^{-1},Y_{d,V}].\]
\end{corollary}
\subsection{Comparison with Sinha's results for \texorpdfstring{$G=\mathbb{Z}/2$}{G=Z/2}}\label{comp}
The description of $MO_*^{\mathbb{Z}/2}$ in \cite[The\-orem 2.4]{S02} is more explicit than ours in Theorem \ref{coro1}. 
In both cases $MO_*^{\mathbb{Z}/2}$ is identified with a subring of $MO_*[e_\sigma, e_\sigma^{-1}, Y_{d,\sigma}].$ Here $\sigma$ denotes the non-tri\-vial one-di\-men\-sio\-nal real re\-presentation of $\mathbb{Z}/2$. 
Also the description of $\mathfrak N_*^{\mathbb{Z}/2}$ in Theorem 2.7 of \cite{S02} is more explicit than ours, but the generators given there can be derived from the pullback property of our Theorem \ref{coro1} and Theorem 2.4 of \cite{S02}. 
\subsection{Real equivariant bordism for \texorpdfstring{$G\neq \mathbb{Z}/2\times\dots\times\mathbb{Z}/2$}{G not a product of Z/2}}\label{count}Theorem \ref{coro1} fails to be true if $G$ is not of the form $\mathbb{Z}/2\times\cdots\times\mathbb{Z}/2.$
For the complex case Hanke shows that his theorem \cite[Theorem 1]{H05} does not hold if $G$ is not of the form $S^1\times \cdots \times S^1$. He gives counterexamples for $G=\mathbb{Z}/n\times\mathbb{Z}/n$ and $G=\mathbb{Z}/n^2$ \cite[Section 4]{H05}. In the real case the situation is similar. There are different ways Theorem \ref{coro1} can fail for $G$ not of the form $\mathbb{Z}/2\times\cdots\times\mathbb{Z}/2.$ 
\begin{propo}\label{toralmono} The homomorphism 
$j_\mathfrak{N}\co \mathfrak{N}_*^G[\mathcal{A}]\to\mathfrak{N}_*^G[\mathcal{A},\mathcal{P}]$ from the Conner--Floyd exact sequence (see Section \ref{CFtDes}) is a monomorphism if and only $G=(\mathbb{Z}/2)^k$ for some $k$. 
\end{propo}
Together with Proposition \ref{kappa1} we immediately get the following. 
\begin{corollary}
 The homomorphism
$\phi_{\mathfrak N}\co \mathfrak N_*^G\to MO_*[e_V^{-1},Y_{d,V}]$ is a monomorphism if and only if $G=(\mathbb{Z}/2)^k$ for some $k$.
\end{corollary}
\begin{proof}[Proof of Proposition \ref{toralmono}]
 Stong proves in \cite[Proposition 14.2, p.\ 75]{S70} that the map
$\iota_\mathfrak N\co \mathfrak N_*^G[\mathcal{P}]\to\mathfrak N_*^G[\mathcal{A}]$ is trivial if and only if $G=(\mathbb{Z}/2)^k$ for some $k$. One direction is proved already in \cite[Proposition 2]{S70a}. Taking this together with the Conner--Floyd exact sequence of the pair $(\mathcal{A},\mathcal{P})$ completes the proof.
\end{proof}
Stong provides an example of a non-zero element in $\mathfrak{N}_3^{\mathbb{Z}/4}$ that is mapped to zero by $j_\mathfrak{N}$.
Using similar techniques as Hanke in \cite{H05}, a counterexample can be constructed, proving for $G=\mathbb{Z}/4$ that the map 
$\iota\circ\phi_{MO}\co  MO_*^{\mathbb{Z}/4}\to MO_*[e_V,e_V^{-1},Y_{d,V}]$ fails to be injective.
\bibliographystyle{gtart}
\bibliography{./Litliste.bib}

\end{document}